\documentclass[12pt,leqno]{amsart}
\usepackage{tikz}
\usetikzlibrary{automata,positioning}
\usepackage{amssymb,amsthm,amsmath,latexsym}
\usepackage{graphicx} % Required for including images
\graphicspath{{figures/}} % Location of the graphics files
\usepackage{booktabs} % Top and bottom rules for table
\usepackage[font=small,labelfont=bf]{caption} % Required for specifying captions to tables and figures
\usepackage{amsfonts, amsmath, amsthm, amssymb} % For math fonts, symbols and environments
\usepackage{wrapfig} % Allows wrapping text around tables and figures

\usepackage{tikz}
\usetikzlibrary{automata,positioning}
\usepackage{amssymb,amsthm,amsmath,latexsym}
\usepackage{graphicx} % Required for including images
\graphicspath{{figures/}} % Location of the graphics files
\usepackage{booktabs} % Top and bottom rules for table
\usepackage[font=small,labelfont=bf]{caption} % Required for specifying captions to tables and figures
\usepackage{amsfonts, amsmath, amsthm, amssymb} % For math fonts, symbols and environments
\usepackage{wrapfig} % Allows wrapping text around tables and figures

\newtheorem{theorem}{\sc Theorem}[section]
\newtheorem{thm}[theorem]{\sc Theorem}
\newtheorem{lem}[theorem]{\sc Lemma}
\newtheorem{defin}[theorem]{\sc Definition}
\newtheorem{prop}[theorem]{\sc Proposition}

 \newtheorem*{thmA}{Theorem A}
 \newtheorem*{thmB}{Theorem B}
 \newtheorem*{thmC}{Theorem C}
 \newtheorem*{thmD}{Theorem D}

\title{On self-similarity of finitely generated torsion-free nilpotent groups}

\author{Adilson Berlatto}
\address{Departamento de Matem\'atica, Universidade Federal de Mato Grosso,
Pontal do Araguaia-MT, 78698-000 Brazil}
\email{(Berlatto) adilson.berlatto@ufmt.br}

\author{Tulio Gentil}
\address{Departamento de Matem\'atica, Universidade Federal do Rio de Janeiro
Rio de Janeiro-RJ, 21941-909 Brazil}
\email{(Gentil) tuliosantosufg@gmail.com}


\begin{document}
\maketitle

\begin{abstract}
This paper aims to investigate the self-similarity property in finitely-generated torsion-free nilpotent groups. We establish connections between geometric equivalence and self-similarity in these groups. Moreover, we show that any 2-generated torsion-free nilpotent group of class 3 has a faithful representation as a transitive self-similar group. Additionally, we provide an example of a 4-generated torsion-free nilpotent group of class 3 that does not admit such a representation. This example, sourced from the work of V. Bludov and B. Gusev in \cite{BG}, resolves a question posed by S. Sidki in \cite{LectureDusseldorf}.
\end{abstract}

\section{Introduction}

A \textit{self-similar} group is a group $G$ that admits a faithful state-closed representation on the regular one-rooted $m$-tree $\mathcal{T}_m$.  If $G$ does not admit such a representation for any $m$ then we say that $G$ is not self-similar.

One approach to generating self-similar groups involves virtual endomorphisms. Let $G$ be a group and $$\mathbf{H=} \left( H_{i} \leqslant G \mid \left[ G:H_{i}\right] =m_{i}\text{ }
\left( \text{ }1\leq i\leq s\right) \right)$$ be a $s$-sequence of subgroups of finite index in $G$ and consider a $s$-sequence
$\mathbf{F}=\left(f_{i}:H_{i}\rightarrow G \,\, \text{}\mid 1\leq i\leq s\right) \text{}$ of homomorphisms (called \textit{virtual endomorphisms}). This setup produces a homomorphism from $G$ to the group $\mathcal{A}_m$ of automorphisms of $\mathcal{T}_m$, where $m=m_1+\dots+m_s$. The kernel $K$ of this homomorphism is the biggest subgroup of the intersection $ \cap _{i=1}^{s}H_{i}$ which is normal in $G$ and $f_i$-invariant for all $i=1,\dots,s$. This subgroup $K$ is referred to as the $\mathbf{F}$-\textit{core of} $\mathbf{H}$.
The group $G$ is self-similar if and only if there exist a pair  $(\mathbf{H},\mathbf{F})$ with $K$ trivial; if in addition, $s=1$ for some such pair  $(\mathbf{H},\mathbf{F})$, $G$ is  said to be \textit{transitive} self-similar; otherwise $G$ is called \textit{intransitive} self-similar.

V. Nekrashevych and S. Sidki classified the free abelian groups of finite rank that act as self-similar groups on the binary tree $\mathcal{T}_2$ \cite{NS}. In particular, they demonstrated that a free abelian group of finite rank can be represented as a self-similar subgroup of the automorphism group of the binary tree.

Following P. Hall's notations, we denote the finitely generated torsion-free nilpotent groups of class $c$ by $\mathfrak{T}_c$ and we delete the index $c$ of the notation for finitely generated torsion-free nilpotent group. In \cite{AS}, it was demonstrated that $\mathfrak{T}_2$-groups  exhibit rich self-similarity properties: \textit{if $G$ is a $\mathfrak{T}_2$-group and $H$ is a subgroup of finite index in $G$, then there exist a subgroup $K$ of finite index in $H$ which admits a recurrent endomorphism $f:K \to G$} (we refer to the endomorphism $f$ as \textit{recurrent} if it is an epimorphism) with a trivial $f$-core of $K$. As a consequence, every $\mathfrak{T}_2$-group is transitive and self-similar. Recently, we proved in \cite{BDS} that any free nilpotent group of finite rank has a transitive recurrent self-similar action.

The abundance of self-similar representations in these groups prompted the following definition:

\begin{defin} 
A group $G$ is said to be \textit{strongly self-similar (strongly recurrent self-similar)} if for every finite-index subgroup $H$ of $G$, there exists a finite-index subgroup $K$ of $H$ and a simple virtual endomorphism (simple recurrent virtual endomorphism) $f:K \rightarrow G$.  
\end{defin}

The $\mathfrak{T}_2$-groups and the free nilpotent groups of finite rank are strongly recurrent self-similar groups. Another important example of a strongly recurrent self-similar $\mathfrak{T}_c$-group is $Tr_1(c+1, \mathbb{Z})$, the (upper) unitriangular group of dimension $c+1$ over $\mathbb{Z}$ (refer to Proposition \ref{Unit}). These groups exhibit two properties that we will discuss here and share an intimate relationship with recurrent self-similarity: the existence of isomorphic proper subgroups and geometric equivalence with their Malcev closures.

In \cite{LectureDusseldorf}, S. Sidki asked about the existence of $\mathfrak{T}_3$-groups which are not self-similar (Problem 10). In this context we prove:

\begin{thmA}
    Every 2-generated $\mathfrak{T}_3$-group is strongly recurrent self-similar.
\end{thmA}

Recall that a group $G$ is \textit{compressible} if whenever $H$ is a subgroup of finite
index in $G$ then there exists a subgroup $K$ of finite index in $H$ with $K$ isomorphic
to $G$. So, strongly recurrent self-similar $\mathfrak{T}$-groups are compressible. In other direction, a group $G$ is said to be \textit{cohopfian} if every injective endomorphism of $G$ is an automorphism.

Two groups $G_1$ and $G_2$ are \textit{commensurable} if $G_1 \cap G_2$ has finite index in both $G_1$ and $G_2$. Given a group $G$, let us denote by $G^{\ast}$ its Malcev completion. It is known that $G_1$ and $G_2$ are commensurable if and only if $G_1^{\ast} \simeq G_2^{\ast}$. In \cite{smith}, it was proven that if $G_1$ and $G_2$ are $\mathfrak{T}$-groups such that $G_1^{\ast} \simeq G_2^{\ast}$, then $G_1$ is compressible if and only if $G_2$ is compressible. We extend this result for strongly recurrent self-similar $\mathfrak{T}$-groups as follows:

\begin{thmB}
    Let $G_1$ and $G_2$ be $\mathfrak{T}_c$-groups such that $G_1^{\ast} \simeq G_2^{\ast}$. Then $G_1$ is strongly recurrent self-similar if, and only if, $G_2$  is strongly recurrent self-similar.
\end{thmB}

The notion of geometric equivalence of groups was introduced by B. Plotkin in \cite{BP}. In \cite{BP2}, B. Plotkin, E. Plotkin, and A. Tsurkov proved that $Tr_1(n \mathbb{Z})$ is geometrically equivalent to its Malcev completion. A. Tsurkov \cite{AT}, also demonstrated that every $\mathfrak{T}_2$-group is geometrically equivalent to its Malcev completion. In general, the following result establishes that a strongly recurrent self-similar $\mathfrak{T}$-group is geometrically equivalent to its Malcev completion.

\begin{thmC}
Strongly recurrent self-similar  $\mathfrak{T}$-groups are geometrically equivalent with their Malcev completions.

\end{thmC}

A group acting
on the tree $\mathcal{T}_m$ is \textit{finite-state}  if each of its elements has a finite number of states. In the study conducted by \cite{BS}, it was proved that the affine group $\mathbb{Z}^n \rtimes GL(n, \mathbb{Z})$ is a finite-state  and self-similar, yet it is still unkown if a subgroup of $\mathcal{A}_m$, where $m=2^n$. As a $\mathfrak{T}$-group is isomorphic to a subgroup of $GL(n, \mathbb{Z})$, it follows that every $\mathfrak{T}$-group is a finite-state group. However, there exist $\mathfrak{T}$-groups that do not admit a self-similar representation.

J. Dyer \cite{JD}, constructed a $2$-generated $\mathfrak{T}_6$-group with Hirsh length 9, which does not admit a
faithful transitive self-similar representation. More recently, using Lie algebras constructed by Y. Benoist \cite{YB}, O. Mathieu exhibited a $\mathfrak{T}$-group with a cyclic center and Hirsch length 12 that is not self-similar \cite{OM}.

In \cite{BG}, V. Bludov and B. Gusev introduced
the $\mathfrak{T}_3$-group $N_{3,4}$ generated by 4 elements, the $\mathfrak{T}_4$-group $N_{4,3}$ with 3 generators, and the $\mathfrak{T}_5$-group $N_{5,2}$ generated by 2 elements and proved these groups to be geometrically nonequivalent to their respective Malcev closures yet are strongly incompressible $\mathfrak{T}$-groups. A group $G$ is called strongly incompressible (see, \cite{Noskov}) if, for all $\sigma \in \text{End}(G)$, it holds that $\sigma \in \text{Aut}(G)$ or $Z(G) \leqslant \text{ker}(\sigma)$. In fact, Noskov \cite{Noskov} established that strongly incompressible $\mathfrak{T}$-groups are not geometrically equivalent to their Malcev closures. This already implies that strongly incompressible $\mathfrak{T}$-groups are not strongly recurrent self-similar. However, concluding that these groups are not self-similar requires a few additional considerations.

%In their work \cite{BG}, V. Bludov and B. Gusev introduced the $\mathfrak{T}_3$-group $N_{3,4}$ generated by 4 elements, the $\mathfrak{T}_4$-group $N_{4,3}$ with 3 generators, and the $\mathfrak{T}_5$-group $N_{5,2}$ generated by 2 elements, studying their geometric equivalences. In a related study, non-compressible groups are discussed in \cite{Noskov}: a group $G$ is called strongly incompressible if, for all $\sigma \in \text{End}(G)$, it holds that $\sigma \in \text{Aut}(G)$ or $Z(G) \leqslant \text{ker}(\sigma)$. In \cite{BG}, V. Bludov and B. Gusev proved that the aforementioned groups are not geometrically equivalent to their respective Malcev closures and are strongly incompressible $\mathfrak{T}$-groups. In fact, Noskov \cite{Noskov} established that strongly incompressible $\mathfrak{T}$-groups are not geometrically equivalent to their Malcev closures. This already implies that strongly incompressible $\mathfrak{T}$-groups are not strongly recurrent self-similar. However, concluding that these groups are not self-similar requires a few additional considerations.

Utilizing the results from \cite{BG}, we address the problem posed by S. Sidki in \cite{LectureDusseldorf}, providing a detailed proof that the group $N_{3,4}$ cannot admit a faithful self-similar representation, whether transitive or not.

\begin{thmD}
The group $N_{3,4}$ cannot admits a faithful state-closed representation in $\mathcal{T}_m$ for any $m$.
\end{thmD}
 
An important consequence of the theorem is noteworthy: the minimal class $c$ for which  examples of non-self-similar $\mathfrak{T}_c$-groups exist, whether transitive or not, is 3. However, an open question remains regarding the smallest number of generators for which a $\mathfrak{T}_3$-group is not self-similar.

We conclude this section with two problems:

\noindent \textbf{Problem 1:} Are there $\mathfrak{T}_3$-groups generated by only three elements that are not self-similar?

\noindent \textbf{Problem 2:} Can a co-hopfian $\mathfrak{T}$-group, which is not geometrically equivalent to its Malcev closure, be self-similar?

\section{Preliminaries}
\subsection{Self-similar groups and virtual endomorphisms}
Let $Y=\{0,\dots,m-1\}$ be a finite alphabet with $m\geq 1$ letters. The set of finite words $Y^*$ over $Y$ has a natural structure
of a rooted $m$-ary tree, denoted by $\mathcal{T}_m$. 

\noindent The automorphism group $\mathcal{A}_{m} ={Aut}\left( \mathcal{T}_m\right) $
of $\mathcal{T}_{m}$ is isomorphic to the restricted wreath product
recursively defined as $\mathcal{A}_{m}=\mathcal{A}_{m}\wr Perm(Y)$. An automorphism $\alpha $ of $%
\mathcal{T}_{m}$ has the form $\alpha =(\alpha _{0},\dots,\alpha _{m-1})\sigma
(\alpha )$, where the state $\alpha _{i}$ belongs to $\mathcal{A}_{m}$ and $\sigma :\mathcal{A}_{m}\rightarrow Perm(Y)$ is the permutational
representation of $\mathcal{A}_{m}$ on $Y$, the first level of the tree $%
\mathcal{T}_{m}$.
The action of $\alpha =(\alpha _{0},\dots,\alpha _{m-1})\sigma
(\alpha ) \in \mathcal{A}_m$ on a word $y_{1}y_{2}\cdots y_{n}$ of length $n$ is given recursivelly by  $\left( y_{1}\right) ^{\sigma (\alpha
)}\left( y_{2}\cdots y_{n}\right) ^{\alpha _{y_{1}}}$.

\noindent Given an element $\alpha=(\alpha _{0},\dots,\alpha _{m-1})\sigma
(\alpha )$ belongs to  ${Aut}\left( \mathcal{T}_m\right) $, the set of automorphisms%
\begin{equation*}
Q(\alpha )=\{\alpha\}\cup Q(\alpha _{0}) \cup Q(\alpha _{1}) \cup \cdots \cup Q(\alpha _{m-1})
\end{equation*}%
is called the set of \textit{states} of $\alpha $. A subgroup 
$G$ of $\mathcal{A}_{m}$ is \textit{self-similar} (or \textit{state-closed}) if $Q(\alpha )$ is a subset of $G$ for all $\alpha $ in $G$ and is \textit{transitive} if its action on the first level of the tree is transitive. 

%An element $\alpha$ belongs to $G$ is said to be \textit{finite state} if $Q(\alpha)$ is finite; the group $G$ is \textit{finite state} if every element of $G$ is finite state. 

A \textit{virtual endomorphism} of an abstract group $G$ is a homomorphism $f:H \to G$ from a subgroup $H\leqslant G$ of finite index to $G$. Let $G$ be a group and consider the $s$-sets
\begin{equation*}
\mathbf{H=} \left( H_{i} \leqslant G \mid \left[ G:H_{i}\right] =m_{i}\text{ }
\left( \text{ }1\leq i\leq s\right) \right) \text{,}
\end{equation*}
\begin{equation*}
\mathbf{m=}\left( m_{1},\dots,m_{s}\right) ,\text{ }m=m_{1}+\cdots+m_{s},
\end{equation*}
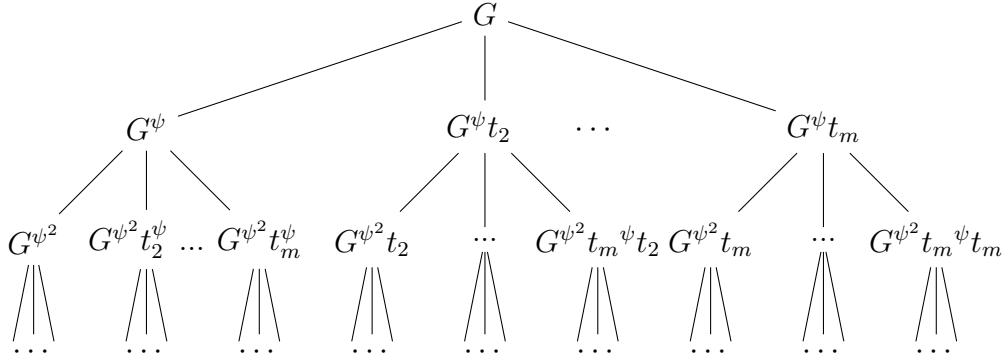
\begin{equation*}
\mathbf{F}=\left\{ f_{i}:H_{i}\rightarrow G \,\, \text{virtual endomorphisms}\mid 1\leq i\leq s\right\} \text{;
}
\end{equation*}
we will call  $\left( \mathbf{m},\mathbf{H,F}\right) $ a $G$-data or data for $G$. The $\mathbf{F}$-core of $\mathbf{H}$ is the biggest subgroup of $ \cap _{i=1}^{s}H_{i}$ which normal in $G$ and $f_i$ invariant for all $i=1,\cdots,s$.

In \cite{DSS} the following approach was constructed to produce self-similar groups, transitive or not. 

\begin{prop}
A subgroup $G$ of $\mathcal{A}_m$ admits a faithful state-closed
representation if and only if there exist a $G$-data $\left( \textbf{m},\textbf{H,F}\right) $ such that the $\mathbf{F}$-core of $\mathbf{H}$  is trivial, that is,
\begin{equation*}
\langle K\leqslant \cap _{i=1}^{s}H_{i}\mid K\vartriangleleft
G,K^{f_{i}}\leqslant K,\forall i=1,\dots,s\rangle =1\text{.}
\end{equation*}
\end{prop}

\subsection{Coset tree}

Let $G$ be a group, $H$ a proper subgroup of $G$ of finite index $m$ and $\psi: G \longrightarrow H$ an isomorphism. Consider $T = \{ t_j \; | \; 1 \leq j \leq m\}$ a right transversal for  $H=G^{\psi}$ in $G$, with  $1 \in T$. Then $T^{\psi^i} = \{ t_j^{\psi^i} \; | \; 1 \leq j \leq m \}$ is a right transversal for  $G^{\psi^{i+1}}$ in $G^{\psi^i}$, for each $i \geq 1$. So, we produce an infinite chain of subgroups of $G$:
\begin{equation}
G \geq G^{\psi} \geq G^{\psi^2} \geq \cdots \geq G^{\psi^i} \geq G^{\psi^{i+1}} \geq \cdots.
\label{cadeiapsi}
\end{equation}
This procedure defines a regular tree $T_G$ with root $G$ and vertices  
$$
G^{\psi^n} t_{i_n}^{\psi^{n-1}} t_{i_{n-1}}^{\psi^{n-2}} \cdots t_{i_2}^{\psi} t_{i_1}, 
$$
where $n \geq 1$ and $1 \leq i_1, \ldots , i_n \leq m$. The edges are defined by inclusion over the vertices. This tree is named \textit{Coset Tree} of $G$ (see \cite{Sidki}) induced by the chain (\ref{cadeiapsi}). It can be visualized in figure below.

\begin{figure}
  \label{cosettree}
\begin{center}
\begin{tikzpicture}
[level distance=1.5cm,
  level 1/.style={sibling distance=4.5cm},
  level 2/.style={sibling distance=1.5cm},
  level 3/.style={sibling distance=0.3cm}]
  \node {$G$}
    child {node {$G^{\psi}$}
      child {node {{$G^{{\psi}^2}$}}
      child {node {$ $}}
      child {node {$ \cdots$}}
      child {node {$ $}}}
      child {node {{$G^{{\psi}^2}t_2^{\psi} \,\,...$}}
      child {node {$ $}}
      child {node {$\cdots $}
      }
      child {node {$ $}}}
      child {node {{$G^{{\psi}^2}t_m^{\psi}$}}
        child {node {$ $}}
        child {node {$ \cdots $}}
      child {node {$ $}}
    }
    }
    child {node {\,\,\,\,\,\,\,\,\,\,\,\,\,\,\,\,\,\,$G^{\psi}t_2 \,\,\,\,\,\,\,\,\,\,\, \cdots$}
       child {node {{$G^{{\psi}^2}t_2$}}
      child {node {$ $}}
      child {node {$ \cdots$}}
      child {node {$ $}}}
      child {node {{$ ...$}}
      child {node {$ $}}
      child {node {$\cdots $}
      }
      child {node {$ $}}}
      child {node {{$G^{{\psi}^2}{t_m}^{\psi}t_2$}}
        child {node {$ $}}
        child {node {$ \cdots $}}
      child {node {$ $}}
    }
    }
    child {node {$G^{\psi}t_m$}
     child {node {{$G^{{\psi}^2}t_m$}}
      child {node {$ $}}
      child {node {$ \cdots$}}
      child {node {$ $}}}
      child {node {{$...$}}
      child {node {$ $}}
      child {node {$\cdots $}
      }
      child {node {$ $}}}
      child {node {{$G^{{\psi}^2}{t_m}^{\psi}t_m$}}
        child {node {$ $}}
        child {node {$ \cdots $}}
      child {node {$ $}}
    }
    };
  
\end{tikzpicture}
\end{center}
\end{figure}

The group $G$ acts on $T_G$ by right multiplication and this action produces a transitive self-similar representation $\lambda: G \longrightarrow Aut(T_G)$, defined by
$$
g^{\lambda} = 
\left( 
\left(
t_j g t_{j'}^{-1}
\right)^{ \psi^{-1} \lambda}
\right)_{1 \leq j \leq m}
\sigma(g),
$$
where $j' = (j) \sigma(g)$ and $\sigma$ is the action on the transversal $T$. Observe that $ker \lambda \leqslant \bigcap_{i \geq 0} G^{\psi^i}$. So, when  $\bigcap_{i \geq 0} G^{\psi^i} = 1$, the action is faithful.

\section{Strongly Self-similar $\mathfrak{T}$-groups}

It is known that $\mathfrak{T}_2$-groups and free nilpotent groups of finite rank are strongly recurrent self-similar (see \cite{AS} and \cite{BDS}). Furthermore, strongly recurrent self-similar $\mathfrak{T}$-groups are compressible since recurrent virtual endomorphisms of $\mathfrak{T}$-groups are isomorphisms. Proposition \ref{Unit} and Theorem \ref{T_3_2g} provide additional examples of strongly recurrent self-similar $\mathfrak{T}$-groups. To prove them, we require the following lemmas:

\begin{lem} {\rm (\cite{Segal}, Theorem 5)} \label{solbfinite}
Let $n$ be a positive integer. For $0 \neq m \in \mathbb{Z}$, write 
$K_m = \{ g \in GL_n(\mathbb{Z}) \; | \; g \equiv I_n \,\, mod\,\, m \}$.   Let $G$ be a soluble-by-finite subgroup of $GL_n(\mathbb{Z})$. Then every subgroup of finite index in $G$ contains $G \cap K_m$, for some $m \neq 0$.
\end{lem}

\begin{lem} {\rm (\cite{smith}, Lemma 6)} \label{HG'}
Let $G$ be a  $\mathfrak{T}$-group and $H \leqslant G$. If $[G:HG']$ is finite, then  $[G:H]$ is finite.
\end{lem}

\begin{lem}  {\rm (\cite{Baumslag}, Theorem 2.5)} \label{int}
Let $G$ be a $\mathfrak{T}$-group and $H$ an isolated subgroup of $G$. Then, for every prime $p$, we have
$$
\displaystyle\bigcap_{i \geq 1} G^{p^i} H = H.
$$
 \end{lem}

\begin{prop} \label{Unit}
The group $Tr_1(n, \mathbb{Z})$, the upper unitriangular group of dimension $n$ over $\mathbb{Z}$, is strongly recurrent self-similar.
\end{prop}
\begin{proof}
Let $m \in \mathbb{N}^*$ and $D_m = diag(1,m,m^2, \ldots , m^{n-1})$. Then $$Tr_1(n, \mathbb{Z})^{D_m} \leqslant Tr_1(n, m\mathbb{Z}).$$ Since the elements of $Tr_1(n, \mathbb{Z})^{D_m}$ has the form 
$$
\left(
\begin{array}[pos]{ccccc}
	1 & a_{12}m & a_{13}m^2 & \cdots & a_{1n}m^{n-1} \\
	0 & 1       & a_{23}m   & \cdots & a_{2n}m^{n-2}  \\
	\vdots & \vdots & \vdots & \cdots& \vdots \\
	0 & 0 & 0 & \cdots & a_{n-1,n}m \\
	0 & 0 & 0 & \cdots & 1
\end{array}
\right),
$$
it is easy to see that the homomorphism $f: Tr_1(n, \mathbb{Z})^{D_m} \rightarrow Tr_1(n, \mathbb{Z})$
defined by $$x \mapsto x^{{D_m}^{-1}},\,\,\,\, x \in Tr_1(n, \mathbb{Z})^{D_m} $$ 
is a bijective simple virtual endomorphism of $Tr_1(n, \mathbb{Z})$. Now, if $H \leqslant G$ has finite index, by Lemma \ref{solbfinite}, there exists $k \in \mathbb{N}^*$ such that $Tr_1(n, k\mathbb{Z}) \leqslant H$. Conjugation by $D_k$ gives the desired result.
\end{proof}

\begin{thm}\label{T_3_2g}
Every 2-generated $\mathfrak{T}_3$-group is strongly recurrent self-similar.
\end{thm}
\begin{proof}
Let $G$ be a $\mathfrak{T}_3$-group generated by $a$ and $b$ and $H$ be a finite-index subgroup of $G$.  It is known that every element of $G$ is written as a product of basic commutators. So, the relations in $G$ have the form 
$$
a^m b^n [a,b]^{k_1} [a,b,a]^{k_2} [a,b,b]^{k_3}, 
$$
where $m,n,k_1, k_2, k_3 \in \mathbb{Z}$.

Suppose that $m \neq 0$. If $K = \left\langle b \right\rangle$, the right cosets of $KG'$ in $G$ have the form $KG'a^i$, where $0 \leq i \leq m-1$. Thus $KG'$ has finite index in $G$ and, by Lemma \ref{HG'}, $[G:K]$ is finite. But it is not possible, once $K$ and $G$ have nilpotency classes 1 and 3, respectively. So $m = 0$. The same argument produces $n = 0$. The relations of $G$ assume the form:
$$
r_1(a,b) = [a,b]^{k_{11}} [a,b,a]^{k_{12}} [a,b,b]^{k_{13}} \mbox{ or } r_2(a,b) =[a,b,a]^{k_{14}} [a,b,b]^{k_{15}},
$$
for some $k_{11}, k_{12}, k_{13}, k_{14}, k_{15} \in \mathbb{Z}$.

There exists a positive integer $d$ such that $$a^d, b^d, [a,b]^d,  [a,b,a]^d,  [a,b,b]^d \in H.$$
Consider $m_1, m_2, m_3 \in \mathbb{N}^*$ and the elements $A = a^{m_1}[a,b]^{m_2}$ and $B = b^{m_1}[a,b]^{m_3}$ of $G$. It was proved in \cite{Makanin}{ (Theorem 2)} that $\langle A, B \rangle$ has finite index in $G$ and the function $\psi$  defined in $\{ a,b \}$ by $a^{\psi} = A$ and $b^{\psi} = B$ extends to an isomorphism $\psi:G \longrightarrow \left\langle A, B \right\rangle$ when 
$$m_1 = k_{11} d,
m_2 = k_{13}d - dm_1k_{13} - \frac{m_1(m_1-1)}{2} $$ 
 and 
$$
m_3 = -k_{12}d + dm_1k_{12} + \frac{m_1(m_1-1)}{2}.
$$
Now observe that
$$
[A,B] =  
[a,b]^{-m_1^2} [a,b,a]^{m_1(- m_3 + \frac{m_1(m_1-1)}{2})} [a,b,b]^{m_1(m_2 - \frac{m_1(m_1-1)}{2})}
\in G^{m_1},
$$
which give us that $a^{\psi^2} =  A^{m_1}[A,B]^{m_2}$ and  $b^{\psi^2} =  B^{m_1}[A,B]^{m_3}$ belongs to $G^{m_1}$. 

Let $\lambda: G \longrightarrow Aut(T_G)$ be the recurrent self-similar representation defined by chain (\ref{cadeiapsi}). By Lemma \ref{int}, we have that
$$
ker \lambda \leqslant \bigcap_{i \geq 0} G^{\psi^i} \leqslant \bigcap_{k \geq 0} G^{\psi^{2k}} \leqslant \bigcap_{k \geq 0} G^{{m_1}^k} = 1.$$
Therefore $\lambda$ is faithful and $G$ is a strongly recurrent self-similar group. \end{proof}

%\section{????????????$T_3$ grupo 3 gerado?????????????}

\section{Malcev Completions in Strongly Recurrent Self-similar Groups}

To prove Theorem B, some results will be necessary:

\begin{lem}
    {\rm (\cite{Segal}, Proposition 2)} \label{nc}
If $G$ is a nilpotent group of class at most $c$ and $n$ is a positive integer, then every element of $G^{n^c}$ is a $n$-power of an element of $G$.
\end{lem}

Given a group $G$, let us denote by $G^{\ast}$ its Malcev completion, and $G^{1/n} = \left\langle u \in G^{\ast} \,| \, u^n \in G \right\rangle$. Lemma \ref{nc} implies that $(G^{n^c})^{1/n} \leqslant G$: the nth roots of elements in $G^{n^c}$ belong to $G$.

\begin{lem}
    {\rm (\cite{GSS}, Lemma 4.10)} \label{index}
Let $G$ be a $\mathfrak{T}$-group and $\alpha \in Aut(G^{\ast}) \cap End(G)$. Then 
$$[G: G^{\alpha}] = |det(\alpha)|.$$
\end{lem}

\begin{lem}
    {\rm (\cite{BG}, Corollary 1.2)} \label{Mn}
Suppose that $G$ is a torsion-free nilpotent group and $M$ a finitely generated subgroup of $G^{\ast}$. Then there exists a positive integer $n = n(M)$ such that $M^n \leqslant G$.
\end{lem}

Note that, according to Lemma \ref{Mn}, any two finite-index subgroups of a $\mathfrak{T}$-group have nontrivial intersection since each one possesses powers from the other.

\begin{thm}
     \label{strong2}
Let $G_1$ and $G_2$ be $\mathfrak{T}_c$-groups such that $G_1^{\ast} \simeq G_2^{\ast}$. Then $G_1$ is strongly recurrent self-similar if, and only if, $G_2$  is strongly recurrent self-similar.
\end{thm}

\begin{proof}
   We can consider $G_1, G_2 \leqslant G_1^{\ast}$. Let $H_2$ be a subgroup of $G_2$ with finite index. Since $G_1^{\ast} \simeq G_2^{\ast}$, $G_1$ and $G_2$ are commensurable. Consequently, $G_1, G_2, G_1 \cap G_2$, and $G_1 \cap H_2$ have finite indices in $G = \left\langle G_1, G_2 \right\rangle$. Let $m = [G:\text{core}_G(G_1 \cap H_2)]$. Thus, we have $G \leqslant {(G_1 \cap H_2)}^{1/m} \leqslant {G_1}^{1/m}$.

    Suppose that $G_1$ is strongly recurrent self-similar. As $H_1= {(G_1 \cap H_2)}^{m^c}$ has finite index in $G_1$, there exists $K_1 \leqslant H_1$ of finite index and a simple recurrent virtual endomorphism $f:K_1 \rightarrow G_1$, which is an isomorphism.

Let $\phi$ be the automorphism of $G_1^{\ast}$ induced by $f$. Since $G_1 = K_1^f = K_1^{\phi}$, it follows that $G_1^{{\phi}^{-1}} = K_1$. Now, observe that ${\phi}^{-1} \in \text{Aut}(G_1^{\ast})$, and by Lemma \ref{nc},
$$
G_2^{{\phi}^{-1}} 
\leqslant G^{{\phi}^{-1}} 
\leqslant ({G_1}^{1/m})^{{\phi}^{-1}} 
\leqslant ({G_1}^{{\phi}^{-1}})^{1/m}$$
$$\leqslant {K_1}^{1/m}
\leqslant {((G_1 \cap H_2)}^{m^c})^{1/m}
\leqslant G_1 \cap H_2 \leqslant G_2.
$$
Thus, $[G_2: G_2^{\phi^{-1}}]$ is finite, as implied by Lemma \ref{index}.

Let us prove that $G_2$ is strongly recurrent self-similar. Observe that $G_2^{{\phi}^{-1}} \leqslant G_1 \cap H_2 \leqslant G_2$, which provides that the restriction of $\phi$ to $G_2^{{\phi}^{-1}}$ is a recurrent virtual endomorphism of $G_2$.

%{\bf (ALERTA!!!!! Nesse caso $\phi$ restrito será a id.!!!!)}

Let $N$ be a subgroup of $G_2^{{\phi}^{-1}}$, normal in $G_2$ such that $N^{\phi} \leqslant N$. There exists a positive natural $d$ such that $(G_2^{{\phi}^{-1}})^d \leqslant K_1$. We have that  $N^{d}  \leqslant K_1$ and ${N^d}^{\phi} \leqslant N^d$. Consider $N_1 = N^d \cap Z(K_1)$. Observe that 
$$
N_1^{\phi} 
\leqslant {N^d}^{\phi} \cap Z(K_1)^{\phi}
\leqslant N^d \cap  Z(G_1) = N^d \cap K_1 \cap Z(G_1) =  N^d \cap Z(K_1)  = N_1.
$$ 
Furthermore,
$$
N_1= N^d \cap Z(K_1) = N^d \cap  Z(G_1).
$$
Therefore $N_1$ is a subgroup of $K_1$, normal in $G_1$ and $N_1^f =  N_1^{\phi} \leqslant N_1$. As $f$ is simple, $N_1= 1$. Again we have
$$
1 = N_1=  N^d \cap  Z(G_1) = N^d \cap G_2^{\phi^{-1}}  \cap Z(G_1) = N^d \cap Z(G_2^{\phi^{-1}}) = N^d \cap Z(G_2), 
$$ 
which occur only if $N^d=1$, because  $Z(G_2)$ intercepts non-trivially every  proper normal subgroup of $G_2$.  Then $N= 1$ and $G_2$ is strongly recurrent self-similar.
\end{proof}

\subsection{Geometric Equivalence}

The concept of geometrical equivalence of groups  arises from a Galois type correspondence between subsets of homomorphisms of a free group and subsets of the free group. Let $F$ be a free group of finite rank and $G$ a group. If $S \subset F$, consider the set
$$
\overline{S} = \{ \sigma \in \mbox{Hom}(F,G) \; | \; S \subset \mbox{ker} \sigma \}. 
$$
For $T \subset \mbox{Hom}(F,G)$, consider 
$$
\overline{T} = \bigcap_{ \alpha \in T} \mbox{ker} \alpha.
$$
From these sets, we obtain 
$$
\overline{\overline{S}} = \bigcap_{ \sigma \in \overline{S}} \mbox{ker} \sigma \;\; \mbox{and} \;\;
%$$
%$$
\overline{\overline{T}} = \{ \sigma \in \mbox{Hom}(F,G) \; | \; \overline{T} \subset \mbox{ker} \sigma \}. 
$$

\begin{defin}
    {\rm
The subset $S \subset F$ is called $G-closed$ if $S = \overline{\overline{S}}$. Equivalently, $T \subset$ Hom$(F,G)$ is $G-closed$ if $T = \overline{\overline{T}}$. Two groups $G_1$ and $G_2$ are $geometrically$ $equivalent$ if, for any free group of finite rank $F$ and any subset $S$ of $F$, we have that $S$ is $G_1$-closed if, and only if, $S$ is $G_2$-closed.
}
\end{defin}

Its easy to see that isomorphic groups are geometrically equivalent. The following facts proved in \cite{BG} are relatively close to strongly recurrent self-similar  $\mathfrak{T}$-groups:

\begin{lem} {\rm (\cite{BG}, Assertion 1.1)} \label{KHG}
Suppose that a group $K$ is a subgroup of a group $H$, which is a subgroup of a
group $G$ ($K \leqslant H \leqslant G$) and the groups $K$ and $G$ are geometrically equivalent. Then the group $H$
is geometrically equivalent to the groups $K$ and $G$.
\end{lem}

\begin{lem} {\rm (\cite{BG}, Theorem 3.1)} \label{GHn} 
Let $G$ be a torsion free nilpotent group and let $G^{\ast}$ be its Malcev completion.
The groups $G$ and $G^{\ast}$ are geometrically equivalent if and only if the group $G$ is geometrically
equivalent to the subgroup $G^n$ for each natural $n$.
\end{lem}

Now we can prove the following result:
\begin{thm}
Strongly recurrent self-similar  $\mathfrak{T}$-groups are geometrically equivalent with their Malcev completions.
\end{thm}
\begin{proof}
Let $G$ be strongly recurrent self-similar  $\mathfrak{T}$-group. As $G^n$ has finite index in $G$, there exists a subgroup $K$ of $G^n$ of finite index and a simple recurrent virtual endomorphism $f:K \rightarrow G$. Thus $K \leq G^n \leq G$ and $K$ and $G$ are isomorphic, so geometrically equivalent. By Lemma \ref{KHG}, $G$ is geometrically equivalent to  $ G^n$, for every natural $n$. By Lemma \ref{GHn}, the result follows. 
\end{proof}

\section{The group ${N_{3,4}}$}

We introduce here the aforementioned
$N_{3,4}$ group of Bludov-Gusev and details about its commutator structure established in  \cite{BG}. $N_{3,4}$ is a nilpotent group of class $3$ generated by $a,b,c,d$ and defining relations:

\begin{equation} \label{1}
[a, [a, b]] = [a, b, b] = [a, [c, d]] = [a, d, d] = 1,
\end{equation}
\begin{equation} \label{2}
[b, [b, c]] = [b, [c, d]] = [c, [c, d]] = [c, d, d] = 1,
\end{equation}
\begin{equation} \label{3}
[a, [b, c]]\cdot [a, c, b] = [a, [b, d]] \cdot [a, d, b] = 1,
\end{equation}
\begin{equation}
[a, [a, d]] \cdot [b, c, c]^{-1} = [a, [a, c]] \cdot [b, [b, d]]^{-1} = 1,
\end{equation}
\begin{equation}\label{4}
[a, c, b] \cdot [b, d, d]^{-1} = [a, c, c] \cdot [b, d, c]^{-1} = [a, d, b] \cdot [a, d, c]^{-1} = 1,
\end{equation}
\begin{equation}\label{5}
[a, b] \cdot [a, c, c]^{-1} = [c, d] \cdot [a, [a, c]]^{-1} = 1.
\end{equation}

\subsection{Facts about the group $\mathbf{N_{3,4}}$} %We list below some facts about the group $G=N_{3,4}$ which were established in \cite{BG}. 

\begin{enumerate}
    \item [1.] The sets 
    $$\{[a,c], [a,d], [b,c], [b,d]\}$$  and  
$$\{[a,b], [c,d], [a,[a,d]], [a,[b,c]], [a,[b,d]]\}$$
form a basis of the free abelian groups $Z_2(G)/Z(G)$ and $Z(G)$, respectively.
    \item[2.] $G$ is a $\mathfrak{T}_3$-group.
    \item[3.]  The upper central series
    $$1=Z_0(G) \leqslant Z_1(G) \leqslant Z_2(G) \leqslant Z_3(G)=G$$
  and the lower central series
    $$1=\gamma_4(G)\leqslant \gamma_3(G) \leqslant \gamma_2(G)\leqslant \gamma_1(G)=G$$
    coincide.  \\
\end{enumerate}

\subsection{virtual endomorphisms of $\mathbf{{N_{3,4}}}$}

Consider $G = N_{3,4}$, $H$ a finite index subgroup of $G$ and $f: H \to G$ a virtual endomorphism of $G$. Let $m,n,k,j$ be positive integers such that 
$$a_1 = a^m, b_1 = b^n, c_1 = c^k, d_1 = d^j \in H.$$ 
From now on we will denote by $K$ the group generated by $a_1,b_1,c_1$ and $d_1$, which has finite index in $G$ (see \cite{BG}, Lemma 1.1). Clearly $K$ satisfies the relations
\begin{equation} \label{add}
[a_1, [a_1, b_1]] = [a_1, b_1, b_1] = [a_1, [c_1, d_1]] = [a_1, d_1, d_1] = 1, 
\end{equation}
\begin{equation}\label{bbc}
[b_1, [b_1, c_1]] = [b_1, [c_1, d_1]] = [c_1, [c_1, d_1]] = [c_1, d_1, d_1] = 1 \;\; \mbox{and}\;\;
\end{equation}
\begin{equation}
[a_1, [b_1, c_1]]\cdot [a_1, c_1, b_1] = [a_1, [b_1, d_1]] \cdot [a_1, d_1, b_1] = 1.
\end{equation}
Some calculations produces:
\begin{equation}
[a_1, [a_1, d_1]]^{nk^2} = [b_1, c_1, c_1]^{m^2j}, 
\end{equation}
\begin{equation}
 [a_1, [a_1, c_1]]^{n^2j} =  [b_1, [b_1, d_1]]^{m^2k},
\end{equation}
\begin{equation} \label{acb-bdd}
[a_1, c_1, b_1]^{j^2} = [b_1, d_1, d_1]^{mk},
\end{equation}
\begin{equation} \label{acc-bdc}
[a_1, c_1, c_1]^{nj} = [b_1, d_1, c_1]^{mk}, 
\end{equation}
\begin{equation}\label{adb-adc}
[a_1, d_1, b_1]^k = [a_1, d_1, c_1]^n, 
\end{equation}
\begin{equation} \label{ab-acc}
[a_1, b_1]^{k^2} = [a_1, c_1, c_1]^{n} \;\; \mbox{and} \;\; [c_1, d_1]^{m^2} = [a_1, [a_1, c_1]]^j.
\end{equation}

\noindent Once $K^f \leqslant  G$, we can write 
$$
{a_1}^f = a^{n_{11}} b^{n_{12}} c^{n_{13}} d^{n_{14}} u_1; 
$$
$$
{b_1}^f = a^{n_{21}} b^{n_{22}} c^{n_{23}} d^{n_{24}} u_2;
$$
$$
{c_1}^f = a^{n_{31}} b^{n_{32}} c^{n_{33}} d^{n_{34}} u_3;
$$
$$
{d_1}^f = a^{n_{41}} b^{n_{42}} c^{n_{43}} d^{n_{44}} u_4;
$$
with $u_1, u_2, u_3, u_4 \in G'$. For $1 \leq i,j \leq 4$, consider the matrix 
$$
N_f = (n_{ij}) = 
\left(
\begin{array}[pos]{cccc}
	n_{11} & n_{12} & n_{13} & n_{14} \\
	n_{21} & n_{22} & n_{23} & n_{24} \\
	n_{31} & n_{32} & n_{33} & n_{34} \\
	n_{41} & n_{42} & n_{43} & n_{44} 
\end{array}
\right)
$$
induced by $f$.

The following lemma was established in \cite{BG}.

\begin{lem}\label{prop1}
If $det(N_f)\neq 0$, then  
$$
{a_1}^f =  a_1  u_1, \;\;\;\;
{b_1}^f =  b_1 u_2, \;\;\;\;
{c_1}^f = c_1  u_3 \;\;\;\; \mbox{and} \;\;\;\; 
{d_1}^f = d_1  u_4,
$$
with $u_1, u_2, u_3, u_4 \in G'$. 
\end{lem}

Let $G$ be a group. An IA-automorphism of $G$ is an automorphism of $G$ which induces in $G/G'$ the identity endomorphism. The set of all IA-automorphism of a group $G$ is denoted by $IA(G)$.

It is known that $IA(G)$ is a normal subgroup of $Aut(G)$ and Inn$(G) \leqslant IA(G)$. Furthermore, if $G$ is a $\mathfrak{T}_c$-group, then $IA(G)$ is a $\mathfrak{T}_{c-1}$-group. Moreover, if $G$ is a nilpotent group, $\gamma_j(IA(G))$ transforms each $\frac{\gamma_i(G)}{\gamma_{i+j}(G)}$ identically (see Theorem 2.1 in \cite{Zyman1}). In particular, if $G$ has nilpotency class $c$ and $\alpha \in IA(G)$, then $x^{\alpha} = x$, for all $ x \in \gamma_c(G)$.

Lemma \ref{prop1} means that the induced homomorphism 
$$\overline{f}: K G'/G' \to G/G'$$
$$xG' \mapsto x^f G'$$
acts like an IA-homomorphism: $ x^f G' = xG'$. This is crucial in the Proposition \ref{prop2}.

\section{Non-self-similarity of $N_{3,4}$}

\begin{prop}\label{prop2}
Suppose that $det(N_f)\neq 0$. Let $f:H \to N_{3,4}$ be a virtual endomorphism of $N_{3,4}$, then $Z(K)^f\leqslant Z(K).$ 
\end{prop}
\begin{proof}
By Lemma \ref{prop1}, $f:H \to N_{3,4}$ is a monomorphism and then $f: K \to K^f$ is a isomorphism. Observe that $G=N_{3,4}$, $K$ and $K^f$ have isomorphic Malcev Completions, because $K$ and $K^f$ have finite index in $G$ (see \cite{Segal}, corollaries 6 and 4). The isomorphism $f: K \to K^f$ extends to an automorphism  $\widetilde{f}: G^* \to G^*$, where $G^*$ is the Malcev completion of $G$ and $\widetilde{f} \big|_K = f$. 

Let $\phi = \widetilde{f} \big|_G :G \to G^*$ be the restriction of $\widetilde{f}$ to $G$. If $g \in G$, there exist a positive integer $t$ such that $g^t \in K$. So we have: 
$$
(g^t)^f G' \;\; = \;\; g^t G' \;\;= \;\; (g G')^t.
$$  
On the other hand, 
$$
(g^t)^f G' \;\; = \;\; (g^t)^{\phi} G' \;\;= \;\; (g^{\phi} G')^t.
$$
Therefore $(g G')^t = (g^{\phi} G')^t$ and, as $G/G'$ is torsion free, $g G'= g^{\phi} G'$. We conclude that $g^{\phi} = g z$, with $z \in G'$, i.e., $G^{\phi} \leqslant G$ and $\phi \in  \;$End$(G)$. In fact $\phi$ is an IA-endomorphism of $G$. By  Lemma 2.4 in \cite{Bludov2}, IA-endomorphisms of nilpotent groups are IA-automorphisms, so $\phi \in $ IA$(G)$.

 In the group $Hol(G) = G\rtimes Aut(G)$, the {\it holomorph} of $G$, it is known that 
 $[\gamma_i(G), \gamma_j(IA(G)] \leqslant \gamma_{i+j}(G)$
 for all naturals $i,j$ (see theorem 7.13 in \cite{tong}). According this fact, we have that $[\gamma_3(G),IA(G)] = 1.$ As $\gamma_3(G) = Z(G)$, $[G:K]$ is finite and $[Z(G),IA(G)] = 1$, it follows that $Z(K) = K \cap Z(G)$ and if $x \in Z(K)$, then $[x,\phi]  = 1$, which implies that $x^{\phi} = x$. Now, as $x \in K$, we obtain $x^f = x^{\phi} = x$. Therefore $Z(K)^f \leqslant Z(K)$.
\end{proof}

The following proposition is a re-reading of the Lemma 3.3 in \cite{BG} and its proof is analogous. 

\begin{prop}\label{prop3}
Let $f:H \to N_{3,4}$ be a virtual endomorphism such that $det(N_f)=0$. Then, there exists $W$ a nontrivial subgroup of $K$ which is normal in $N_{3,4}$ and $W \leqslant Ker(f)$.
\end{prop}

\begin{theorem}
The group $N_{3,4}$ cannot admits a faithful state-closed representation in $\mathcal{T}_m$ for any $m$.
\end{theorem}

\begin{proof}
Let $\left( 
\mathbf{m},\mathbf{H,F}\right)$ be a $N_{3,4}$-data. Since $H_i$ has finite index on $N_{3,4}$ we can assume $K \leqslant \cap_{i=1}^{s}H_i$. 

Assume that there exist $f_i \in \mathbf{F}$ such that 
$det(N_{f_i})=0$. Then, by  Proposition \ref{prop3}, there is a nontrivial subgroup $W$ of $K$ which is normal in $G$ and $W\leqslant Ker(f).$
Define $U=W \cap Z(K)$. Then $U$ is a non-trivial subgroup of $\cap_{i=1}^{s}H_i$ normal in $G$ and $f_i$-invariant for all $i=1,\dots,s$. If each $f_i \in H$ is such that $det(N_{f_i})\neq 0$, then 
 by Proposition \ref{prop2}, $Z(K)^{f_i} \leqslant Z(K),$ and  $Z(K)$ is a non-trivial subgroup of $\cap_{i=1}^{s}H_i$ normal in $G$ which is $f_i$-invariant for all $i=1,\dots,s$. Therefore, $G$ cannot have a faithful state-closed representation. 
\end{proof}

\noindent\textbf{Acknowledgments}

\noindent The authors would like to thank  Said Sidki, and Alex Dantas
for fruitful discussions during the preparation of the paper.

\end{document}